\documentclass[10pt]{amsart} 
\usepackage{graphicx}
\usepackage{amsmath}
\usepackage{amssymb}
\newtheorem{theorem}{Theorem}[section]
\newtheorem{lemma}[theorem]{Lemma}
\newtheorem{corollary}[theorem]{Corollary}

\newtheorem{definition}[theorem]{Definition}

\theoremstyle{definition}
\newtheorem{remark}[theorem]{Remark}

\numberwithin{equation}{section}
\DeclareMathOperator{\RE}{Re} 
\begin{document}
\title[A class of harmonic functions defined by convolution]{A comprehensive class of harmonic functions defined by convolution and its connection with integral transforms and hypergeometric functions}
\author{Sumit Nagpal}
\address{Department of Mathematics,\\
 University of Delhi,\\
 Delhi--110 007,\\
 India}
\email{sumitnagpal.du@gmail.com }
\author{V. Ravichandran}
\address{Department of Mathematics,\\
 University of Delhi,\\
 Delhi--110 007,\\
 India}
\email{vravi68@gmail.com}
\begin{abstract}
For given two harmonic functions $\Phi$ and $\Psi$ with real coefficients in the open unit disk $\mathbb{D}$, we study a class of harmonic functions $f(z)=z-\sum_{n=2}^{\infty}A_nz^{n}+\sum_{n=1}^{\infty}B_n\bar{z}^n$ $(A_n,  B_n \geq 0)$ satisfying
\[\RE \frac{(f*\Phi)(z)}{(f*\Psi)(z)}>\alpha \quad (0\leq \alpha <1, z \in \mathbb{D});\]
$*$ being the harmonic convolution. Coefficient inequalities, growth and covering theorems, as well
as closure theorems are determined. The results obtained extend
several known results as special cases. In addition, we study the class of harmonic functions $f$ that satisfy $\RE f(z)/z>\alpha$ $(0\leq \alpha <1, z \in \mathbb{D})$. As an application, their connection with certain integral transforms and hypergeometric functions is established.
\end{abstract}
\maketitle

\section{Introduction}
Let $\mathcal{H}$ denote the class of all complex-valued harmonic functions $f$ in the open unit disk $\mathbb{D}=\{z \in \mathbb{C}:|z|<1\}$ normalized by $f(0)=0=f_{z}(0)-1$. Such functions can be written in the form $f=h+\bar{g}$, where
\begin{equation}\label{eq1.1}
h(z)=z+\sum_{n=2}^{\infty}A_nz^{n}\quad\mbox{and}\quad g(z)=\sum_{n=1}^{\infty}B_nz^n
\end{equation}
are analytic in $\mathbb{D}$. We call $h$ the analytic part and $g$ the co-analytic part of $f$. Let $\mathcal{S}_{H}$ be the subclass of $\mathcal{H}$ consisting of univalent and sense-preserving functions. In 1984, Clunie and Sheil-Small \cite{cluniesheilsmall} initiated the study of the class $\mathcal{S}_{H}$ and its subclasses.

For analytic functions $\phi(z)=z+\sum_{n=2}^{\infty}A_nz^{n}$ and $\psi(z)=z+\sum_{n=2}^{\infty}A'_nz^{n}$, their convolution (or Hadamard product) is defined as $(\phi*\psi)(z)=z+\sum_{n=2}^{\infty}A_nA'_nz^{n}$, $z\in \mathbb{D}$. In the harmonic case, with $f=h+\bar{g}$ and $F=H+\bar{G}$, their harmonic convolution is defined as $f*F=h*H+\overline{g*G}$. Harmonic convolutions are investigated in \cite{dorff1,dorff2,goodloe,sumit1,sumit2}.

Let $\mathcal{TH}$ be the subclass of $\mathcal{H}$ consisting of functions $f=h+\bar{g}$ so that $h$ and $g$ take the form
\begin{equation}\label{eq1.2}
h(z)=z-\sum_{n=2}^{\infty}A_nz^{n}\quad,\quad g(z)=\sum_{n=1}^{\infty}B_nz^n\quad (A_n,  B_n \geq 0).
\end{equation}
Making use of the convolution structure for harmonic mappings, we study a new subclass $\mathcal{TH}(\Phi_{i},\Psi_{j};\alpha)$ introduced in the following:

\begin{definition}
Suppose that $i,j\in \{0,1\}$. Let the functions $\Phi_{i},\Psi_{j}$ given by
\[\Phi_{i}(z)=z+\sum_{n=2}^{\infty}p_{n}z^{n}+(-1)^{i}\sum_{n=1}^{\infty}q_{n}\bar{z}^n\]
and
\[\Psi_{j}(z)=z+\sum_{n=2}^{\infty}u_{n}z^{n}+(-1)^{j}\sum_{n=1}^{\infty}v_{n}\bar{z}^n\]
are harmonic in $\mathbb{D}$ with $p_n > u_n\geq 0$ $(n=2,3,\ldots)$ and $q_n > v_n\geq 0$ $(n=1,2,\ldots)$. Then a function $f \in \mathcal{H}$ is said to be in the class $\mathcal{H}(\Phi_{i},\Psi_{j};\alpha)$ if and only if
\begin{equation}\label{eq1.3}
\RE \frac{(f*\Phi_i)(z)}{(f*\Psi_j)(z)}>\alpha \quad (z \in \mathbb{D}),
\end{equation}
where $*$ denotes the harmonic convolution as defined above and $0\leq \alpha <1$. Further we define the class $\mathcal{TH}(\Phi_{i},\Psi_{j};\alpha)$ by
\[\mathcal{TH}(\Phi_{i},\Psi_{j};\alpha)=\mathcal{H}(\Phi_{i},\Psi_{j};\alpha)\cap \mathcal{TH}.\]
\end{definition}

The family $\mathcal{TH}(\Phi_{i},\Psi_{j};\alpha)$ includes a variety of well-known subclasses of harmonic functions as well as many new ones. For example
\[\mathcal{TH}\left(\frac{z}{(1-z)^2}-\frac{\bar{z}}{(1-\bar{z})^2},\frac{z}{1-z}+\frac{\bar{z}}{1-\bar{z}};\alpha\right)\equiv \mathcal{TS}_{H}^{*}(\alpha);\]
and
\[\mathcal{TH}\left(\frac{z+z^2}{(1-z)^3}+\frac{\bar{z}+\bar{z}^2}{(1-\bar{z})^3},\frac{z}{(1-z)^2}-\frac{\bar{z}}{(1-\bar{z})^2};\alpha\right)\equiv \mathcal{TK}_{H}(\alpha)\]
are the classes of sense-preserving harmonic univalent functions $f$ which are fully starlike of order $\alpha$ $(0\leq \alpha <1)$ and fully convex of order $\alpha$ $(0\leq \alpha <1)$ respectively (see \cite{jahangiriconvex,jahangiristarlike,sumit2}). These classes were studied by Silverman and Silvia \cite{silverman} for the case $\alpha=0$. Recall that fully starlike functions of order $\alpha$ and fully convex functions of order $\alpha$ are characterized by the conditions
\[\frac{\partial}{\partial \theta}\arg f(r e^{i \theta})> \alpha\quad (0\leq\theta< 2\pi, 0<r<1)\]
and
\[\frac{\partial}{\partial \theta}\left(\arg \left\{\frac{\partial}{\partial \theta}f(r e^{i \theta})\right\}\right)> \alpha,\quad (0\leq\theta< 2\pi, 0<r<1)\]
respectively. In the similar fashion, it is easy to see that the subclasses of $\mathcal{TH}$ introduced by Ahuja \emph{et al.} \cite{ahuja1,ahuja2}, Dixit \emph{et al.} \cite{dixit}, Frasin \cite{frasin}, Murugusundaramoorthy \emph{et al.} \cite{raina} and Yal\c{c}in \emph{et al.} \cite{yalcin1,yalcin3} are special cases of our class $\mathcal{TH}(\Phi_{i},\Psi_{j};\alpha)$ for suitable choice of the functions $\Phi_i$ and $\Psi_j$.

In Section \ref{sec2}, we obtain the coefficient inequalities, growth and covering theorems, as well as closure theorems for functions in the class $\mathcal{TH}(\Phi_{i},\Psi_{j};\alpha)$. In particular, the invariance of the class $\mathcal{TH}(\Phi_{i},\Psi_{j};\alpha)$ under certain integral transforms and connection with hypergeometric functions is also established.

The study of harmonic mappings defined by using hypergeometric functions is a recent area of interest \cite{ahuja,ahuja3,raina}. Let $F(a,b,c;z)$ be the Gaussian hypergeometric function defined by
\begin{equation}\label{eq1.4}
F(a,b,c;z):=\sum_{n=0}^{\infty}\frac{(a)_n(b)_n}{(c)_n(1)_n}z^n,\quad z\in \mathbb{D}
\end{equation}
which is the solution of the second order homogeneous differential equation
\[z(1-z)w''(z)+[c-(a+b+1)z]w'(z)-abw(z)=0,\]
where $a, b, c$ are complex numbers with $c \neq 0,-1,-2,\ldots$, and $(\theta)_n$ is the Pochhammer symbol: $(\theta)_0=1$ and $(\theta)_n=\theta (\theta+1)\ldots(\theta+n-1)$ for $n=1,2,\ldots$. Since the hypergeometric series in \eqref{eq1.4} converges absolutely in $\mathbb{D}$, it follows that $F(a,b,c;z)$ defines an analytic function in $\mathbb{D}$ and plays an important role in the theory of univalent functions. We have obtained necessary and sufficient conditions for a harmonic function associated with hypergeometric functions to be in the class $\mathcal{TH}(\Phi_{i},\Psi_{j};\alpha)$. The well-known Gauss's summation theorem: if $\RE (c-a-b)>0$ then
\[F(a,b,c;1)=\frac{\Gamma (c) \Gamma(c-a-b)}{\Gamma(c-a)\Gamma (c-b)},\quad c \neq 0,-1,-2,\ldots\]
will be frequently used in this paper.

For $\gamma >-1$ and $-1\leq \delta <1$, let $L_\gamma:\mathcal{H}\to \mathcal{H}$ and $G_\delta:\mathcal{H}\to \mathcal{H}$ be the integral transforms defined by
\begin{equation}\label{eq1.5}
L_\gamma[f](z):=\frac{\gamma+1}{z^\gamma}\int_{0}^{z}t^{\gamma-1}h(t)\, dt+\overline{\frac{\gamma+1}{z^\gamma}\int_{0}^{z}t^{\gamma-1}g(t)\, dt},
\end{equation}
and
\begin{equation}\label{eq1.6}
G_\delta[f](z):=\int_0^z \frac{h(t)-h(\delta t)}{(1-\delta)t}\, dt+\overline{\int_0^z \frac{g(t)-g(\delta t)}{(1-\delta)t}\, dt}.
\end{equation}
where $f=h+\overline{g}\in \mathcal{H}$ and $z\in \mathbb{D}$. It has been shown that the class $\mathcal{TH}(\Phi_{i},\Psi_{j};\alpha)$ is preserved under these integral transforms.

For $0\leq \alpha<1$, let
\[\mathcal{TU}_H(\alpha):=\mathcal{TH}\left(\frac{z}{1-z}+\frac{\bar{z}}{1-\bar{z}},z;\alpha\right)\]
The last section of the paper investigates the properties of functions in the class $\mathcal{TU}_H(\alpha)$. Moreover, inclusion relations are obtained between the classes $\mathcal{TU}_H(\alpha)$, $\mathcal{TS}_H^*(\alpha)$ and $\mathcal{TK}_H(\alpha)$ under certain milder conditions.

\section{The Class $\mathcal{TH}(\Phi_{i},\Psi_{j};\alpha)$}\label{sec2}
The first theorem of this section provides a sufficient coefficient condition for a function to be in the class $\mathcal{H}(\Phi_{i},\Psi_{j};\alpha)$.

\begin{theorem}\label{th2.1}
Let the function $f=h+\bar{g}$ be such that $h$ and $g$ are given by \eqref{eq1.1}. Furthermore, let
\begin{equation}\label{eq2.1}
\sum_{n=2}^{\infty}\frac{p_{n}-\alpha u_{n}}{1-\alpha}|A_n|+\sum_{n=1}^{\infty}\frac{q_{n}-(-1)^{j-i}\alpha v_{n}}{1-\alpha}|B_n|\leq 1
\end{equation}
where $0\leq \alpha<1$, $i,j\in \{0,1\}$, $p_n > u_n\geq 0$ $(n=2,3,\ldots)$ and $q_n > v_n\geq 0$ $(n=1,2,\ldots)$. Then $f \in \mathcal{H}(\Phi_{i},\Psi_{j};\alpha)$.
\end{theorem}

\begin{proof}
Using the fact that $\RE w>\alpha$ if and only if $|w-1|<|w+1-2\alpha|$, it suffices to show that
\begin{equation}\label{eq2.2}
|C(z)+(1-2\alpha)D(z)|-|C(z)-D(z)|\geq0,
\end{equation}
where
\[C(z)=(f*\Phi_i)(z)=z+\sum_{n=2}^{\infty}A_n p_n z^n+(-1)^i\sum_{n=1}^{\infty}B_n q_n \bar{z}^{n}\]
and
\[D(z)=(f*\Psi_j)(z)=z+\sum_{n=2}^{\infty}A_n u_n z^n+(-1)^j\sum_{n=1}^{\infty}B_n v_n \bar{z}^{n}.\]
Substituting for $C(z)$ and $D(z)$ in \eqref{eq2.2} and making use of \eqref{eq2.1} we obtain
\begin{align*}
|C(z)&+(1-2\alpha)D(z)|-|C(z)-D(z)|\\
&=\left|2(1-\alpha)z+\sum_{n=2}^{\infty}(p_n+(1-2\alpha)u_n)A_n z^n\right.\\
&\quad\qquad\qquad\left.+(-1)^i\sum_{n=1}^{\infty}(q_n+(-1)^{j-i}(1-2\alpha)v_n)B_n \bar{z}^n\right|\\
&\quad\qquad\qquad-\left|\sum_{n=2}^{\infty} (p_n-u_n)A_n z^n+(-1)^i\sum_{n=1}^{\infty} (q_n-(-1)^{j-i}v_n)B_n\bar{z}^{n}\right|\\
&\geq2(1-\alpha)|z|-\sum_{n=2}^{\infty}(p_n+(1-2\alpha)u_n)|A_n| |z|^n\\
&\quad\qquad\qquad-\sum_{n=1}^{\infty}(q_n+(-1)^{j-i}(1-2\alpha)v_n)|B_n| |z|^n\\
&\qquad\qquad\qquad-\sum_{n=2}^{\infty} (p_n-u_n)|A_n| |z|^n-\sum_{n=1}^{\infty} (q_n-(-1)^{j-i}v_n)|B_n||z|^{n}\\
&=2(1-\alpha)|z|\left[1-\sum_{n=2}^{\infty}\frac{p_n-\alpha u_n}{1-\alpha}|A_n||z|^{n-1}\right.\\
&\quad\qquad\qquad\left.-\sum_{n=1}^{\infty}\frac{q_n-(-1)^{j-i}\alpha v_n}{1-\alpha}|B_n||z|^{n-1}\right]\\
&>2(1-\alpha)|z|\left[1-\sum_{n=2}^{\infty}\frac{p_n-\alpha u_n}{1-\alpha}|A_n|-\sum_{n=1}^{\infty}\frac{q_n-(-1)^{j-i}\alpha v_n}{1-\alpha}|B_n|\right]\geq 0.
\end{align*}
The harmonic mappings
\[f(z)=z+\sum_{n=2}^{\infty}\frac{1-\alpha}{p_n-\alpha u_n}x_n z^n+\sum_{n=1}^{\infty} \frac{1-\alpha}{q_{n}-(-1)^{j-i}\alpha v_{n}}\bar{y}_n \bar{z}^n,\]
where $\sum_{n=2}^{\infty}|x_n|+\sum_{n=1}^{\infty}|y_n|=1$, show that the coefficient bound given by \eqref{eq2.1} is sharp.
\end{proof}

\begin{remark}\label{rem2.2}
In addition to the hypothesis of Theorem \ref{th2.1}, if we assume that $p_n\geq n$ $(n=2,3,\ldots)$ and $q_n \geq n$ $(n=1,2,\ldots)$ then it is easy to deduce that $n(1-\alpha)\leq p_{n}-\alpha u_{n}$ $(n=2,3,\ldots)$ and $n(1-\alpha)\leq q_{n}-(-1)^{j-i}\alpha v_{n}$ $(n=1,2,\ldots)$ so that
\[\sum_{n=2}^{\infty}n|A_n|+\sum_{n=1}^{\infty}n|B_n|\leq\sum_{n=2}^{\infty}\frac{p_{n}-\alpha u_{n}}{1-\alpha}|A_n|+\sum_{n=1}^{\infty}\frac{q_{n}-(-1)^{j-i}\alpha v_{n}}{1-\alpha}|B_n|\leq 1.\]
By \cite[Theorem 1, p.\ 472]{jahangiristarlike}, $f\in \mathcal{S}_{H}$ and maps $\mathbb{D}$ onto a starlike domain.
\end{remark}

Theorem \ref{th2.1} gives a sufficient condition for the harmonic function $\phi_1+\overline{\phi}_2$ to be in the class
$\mathcal{H}(\Phi_{i},\Psi_{j};\alpha)$ where $\phi_1(z)\equiv \phi_1(a_1,b_1,c_1;z)$ and $\phi_2(z)\equiv \phi_2(a_2,b_2,c_2;z)$ are the hypergeometric functions defined by
\begin{equation}\label{hyper1}
\phi_1(z):=z F(a_1,b_1,c_1;z)\quad \mbox{and}\quad \phi_2(z):=F(a_2,b_2,c_2;z)-1.
\end{equation}

\begin{corollary}\label{cor-hyper1}
Let $a_k,b_k,c_k >0$ for $k=1,2$. Furthermore, let
\begin{equation}\label{hyper3}
\sum_{n=2}^{\infty}\frac{p_{n}-\alpha u_{n}}{1-\alpha}\frac{(a_1)_{n-1}(b_1)_{n-1}}{(c_1)_{n-1}(1)_{n-1}}+\sum_{n=1}^{\infty}\frac{q_{n}-(-1)^{j-i}\alpha v_{n}}{1-\alpha}\frac{(a_2)_n(b_2)_n}{(c_2)_n(1)_n}\leq 1
\end{equation}
where $0\leq \alpha<1$, $i,j\in \{0,1\}$, $p_n > u_n\geq 0$ $(n=2,3,\ldots)$ and $q_n > v_n\geq 0$ $(n=1,2,\ldots)$. Then $\phi_1+\overline{\phi}_2 \in \mathcal{H}(\Phi_{i},\Psi_{j};\alpha)$, $\phi_1$ and $\phi_2$ being given by \eqref{hyper1}.
\end{corollary}

The next corollary provides a sufficient condition for $\psi_1+\overline{\psi}_2$ to belong to the class $\mathcal{H}(\Phi_{i},\Psi_{j};\alpha)$ where $\psi_1(z)\equiv \psi_1(a_1,b_1,c_1;z)$ and $\psi_2(z)\equiv \psi_2(a_2,b_2,c_2;z)$ are analytic functions defined by
\begin{equation}\label{hyper2}
\psi_1(z):=\int_0^z F(a_1,b_1,c_1;t)\,dt\quad \mbox{and}\quad \psi_2(z):=\int_0^z (F(a_2,b_2,c_2;t)-1)\,dt.
\end{equation}

\begin{corollary}\label{cor-hyper2}
Suppose that $a_k,b_k,c_k >0$ for $k=1,2$ and
\begin{equation}\label{hyper4}
\sum_{n=2}^{\infty}\frac{p_{n}-\alpha u_{n}}{1-\alpha}\frac{(a_1)_{n-1}(b_1)_{n-1}}{(c_1)_{n-1}(1)_{n}}+\sum_{n=2}^{\infty}\frac{q_{n}-(-1)^{j-i}\alpha v_{n}}{1-\alpha}\frac{(a_2)_{n-1}(b_2)_{n-1}}{(c_2)_{n-1}(1)_n}\leq 1
\end{equation}
where $0\leq \alpha<1$, $i,j\in \{0,1\}$, $p_n > u_n\geq 0$ $(n=2,3,\ldots)$ and $q_n > v_n\geq 0$ $(n=2,3,\ldots)$. Then $\psi_1+\overline{\psi}_2 \in \mathcal{H}(\Phi_{i},\Psi_{j};\alpha)$, $\psi_1$ and $\psi_2$ being given by \eqref{hyper2}.
\end{corollary}

It is worth to remark that Theorems 2.2, 2.4 and 2.11 of \cite{ahuja3} are particular cases of Corollary \ref{cor-hyper1}, while \cite[Theorem 2.8]{ahuja3} follows as a special case of Corollary \ref{cor-hyper2}. We next show that the coefficient condition \eqref{eq2.1} is also necessary for functions in $\mathcal{TH}(\Phi_{i},\Psi_{j};\alpha)$.
\begin{theorem}\label{th2.3}
Let the function $f=h+\bar{g}$ be such that $h$ and $g$ are given by \eqref{eq1.2}. Then $f \in \mathcal{TH}(\Phi_{i},\Psi_{j};\alpha)$ if and only if
\begin{equation}\label{eq2.3}
\sum_{n=2}^{\infty}\frac{p_{n}-\alpha u_{n}}{1-\alpha}A_n+\sum_{n=1}^{\infty}\frac{q_{n}-(-1)^{j-i}\alpha v_{n}}{1-\alpha}B_n\leq 1
\end{equation}
where $0\leq \alpha<1$, $i,j\in \{0,1\}$, $p_n > u_n\geq 0$ $(n=2,3,\ldots)$ and $q_n > v_n\geq 0$ $(n=1,2,\ldots)$.
\end{theorem}

\begin{proof}
The sufficient part follows by Theorem \ref{th2.1} upon noting that $\mathcal{TH}(\Phi_{i},\Psi_{j};\alpha)\subset \mathcal{H}(\Phi_{i},\Psi_{j};\alpha)$. For the necessary part, let $f \in \mathcal{TH}(\Phi_{i},\Psi_{j};\alpha)$. Then \eqref{eq1.3} yields
\begin{align*}
\RE &\left(\frac{(f*\Phi_i)(z)}{(f*\Psi_j)(z)}-\alpha\right)\\
&=\RE\left(\frac{(1-\alpha)z+\sum_{n=2}^{\infty}(p_n-\alpha u_n)A_n z^n+(-1)^i\sum_{n=1}^{\infty}(q_n-(-1)^{j-i}\alpha v_n)B_n\bar{z}^n}{z+\sum_{n=2}^{\infty}A_n u_n z^n+(-1)^j\sum_{n=1}^{\infty}B_n v_n \bar{z}^{n}}\right)\\
&\geq \frac{(1-\alpha)-\sum_{n=2}^{\infty}(p_n-\alpha u_n)A_n |z|^{n-1}-\sum_{n=1}^{\infty}(q_n-(-1)^{j-i}\alpha v_n)B_n|z|^{n-1}}{1+\sum_{n=2}^{\infty}A_n u_n |z|^{n-1}+\sum_{n=1}^{\infty}B_n v_n |z|^{n-1}}>0.
\end{align*}
The above inequality must hold for all $z \in \mathbb{D}$. In particular, choosing the values of $z$ on the positive real axis and letting $z\rightarrow 1^{-}$ we obtain the required condition \eqref{eq2.3}.
\end{proof}

Theorem \ref{th2.3} immediately yields the following three corollaries.

\begin{corollary}\label{cor2.4}
For $f=h+\bar{g}\in \mathcal{TH}(\Phi_{i},\Psi_{j};\alpha)$ where $h$ and $g$ are given by \eqref{eq1.2}, we have
\[A_n\leq\frac{1-\alpha}{p_n-\alpha u_n}\, (n=2,3,\ldots)\quad \mbox{and}\quad B_n\leq \frac{1-\alpha}{q_{n}-(-1)^{j-i}\alpha v_{n}}\, (n=1,2,\ldots);\]
the result being sharp, for each $n$.
\end{corollary}

\begin{corollary}\label{cor-hyper3}
Let $a_k,b_k,c_k >0$ for $k=1,2$ and $\phi_1,\phi_2$ be given by \eqref{hyper1}. Then a necessary and sufficient condition for the harmonic function $\Phi(z)=2z-\phi_1(z)+\overline{\phi_2(z)}$ to be in the class $\mathcal{TH}(\Phi_{i},\Psi_{j};\alpha)$ is that \eqref{hyper3} is satisfied.
\end{corollary}

\begin{corollary}\label{cor-hyper4}
If $a_k,b_k,c_k >0$ for $k=1,2$, then $\Psi(z)=2z-\psi_1(z)+\overline{\psi_2(z)} \in \mathcal{TH}(\Phi_{i},\Psi_{j};\alpha)$ if and only if condition \eqref{hyper4} holds, where $\psi_1,\psi_2$ are given by \eqref{hyper2}.
\end{corollary}

Note that \cite[Theorem 2.6]{ahuja3} is a particular case of Corollary \ref{cor-hyper3}. By making use of Theorem \ref{th2.3}, we obtain the following growth estimate for functions in the class $\mathcal{TH}(\Phi_{i},\Psi_{j};\alpha)$.

\begin{theorem}\label{th2.6}
Let $f \in \mathcal{TH}(\Phi_{i},\Psi_{j};\alpha)$, $\sigma_n=p_n-\alpha u_n$ $(n=2,3,\ldots)$ and $\Gamma_n=q_{n}-(-1)^{j-i}\alpha v_{n}$ $(n=1,2,\ldots)$. If $\{\sigma_n\}$ and $\{\Gamma_n\}$ are non-decreasing sequences, then
\[|f(z)|\leq (1+B_1)|z|+\frac{1-\alpha}{\eta}\left(1-\frac{q_1-(-1)^{j-i}\alpha v_1}{1-\alpha}B_1\right)|z|^2,\]
and
\[|f(z)|\geq (1-B_1)|z|-\frac{1-\alpha}{\eta}\left(1-\frac{q_1-(-1)^{j-i}\alpha v_1}{1-\alpha}B_1\right)|z|^2\]
for all $z \in \mathbb{D}$, where $\eta=\min \{\sigma_2,\Gamma_2\}$ and $B_1=f_{\bar{z}}(0)$.
\end{theorem}

\begin{proof}
Writing $f=h+\bar{g}$ where $h$ and $g$ are given by \eqref{eq1.2}, we have
\begin{align*}
|f(z)|&\leq (1+B_1)|z|+\sum_{n=2}^{\infty} (A_n+B_n)|z|^{n}\\
      &\leq (1+B_1)|z|+\frac{1-\alpha}{\eta}\sum_{n=2}^{\infty} \left(\frac{\eta}{1-\alpha}A_n+\frac{\eta}{1-\alpha}B_n\right)|z|^{2}\\
      &\leq (1+B_1)|z|+\frac{1-\alpha}{\eta}\sum_{n=2}^{\infty} \left(\frac{p_{n}-\alpha u_{n}}{1-\alpha}A_n+\frac{q_{n}-(-1)^{j-i}\alpha v_{n}}{1-\alpha}B_n\right)|z|^{2}\\
      &\leq(1+B_1)|z|+\frac{1-\alpha}{\eta}\left(1-\frac{q_1-(-1)^{j-i}\alpha v_1}{1-\alpha}B_1\right)|z|^2
\end{align*}
using the hypothesis and applying Theorem \ref{th2.3}.

The proof of the left hand inequality follows on lines similar to that of the right hand side inequality.
\end{proof}

The covering result for the class $\mathcal{TH}(\Phi_{i},\Psi_{j};\alpha)$ follows from the left hand inequality of Theorem \ref{th2.6}.

\begin{corollary}\label{cor2.7}
Under the hypothesis of Theorem \ref{th2.6}, we have
\[\left\{w \in \mathbb{C}:|w|<\frac{1}{\eta}(\eta-1+\alpha+(q_1-(-1)^{j-i}\alpha v_1-\eta)B_1)\right\}\subset f(\mathbb{D}).\]
\end{corollary}

Using Theorem \ref{th2.3} it is easily seen that the class $\mathcal{TH}(\Phi_{i},\Psi_{j};\alpha)$ is convex and closed with respect to the topology of locally uniform convergence so that the closed convex hull of $\mathcal{TH}(\Phi_{i},\Psi_{j};\alpha)$ equals itself. The next theorem determines the extreme points of $\mathcal{TH}(\Phi_{i},\Psi_{j};\alpha)$.

\begin{theorem}\label{th2.8}
Suppose that $0\leq \alpha<1$, $i,j\in \{0,1\}$, $p_n > u_n\geq 0$ $(n=2,3,\ldots)$ and $q_n > v_n\geq 0$ $(n=1,2,\ldots)$. Set
\[h_1(z)=z,\quad h_n(z)=z-\frac{1-\alpha}{p_n-\alpha u_n}z^n\, (n=2,3,\ldots)\quad \mbox{and}\]
\[g_n(z)=z+\frac{1-\alpha}{q_{n}-(-1)^{j-i}\alpha v_{n}}\bar{z}^{n}\, (n=1,2,\ldots).\]
Then $f \in \mathcal{TH}(\Phi_{i},\Psi_{j};\alpha)$ if and only if it can be expressed in the form
\begin{equation}\label{eq2.4}
f(z)=\sum_{n=1}^{\infty}(X_n h_n+Y_n g_n)(z),\quad X_n\geq 0,\quad Y_n\geq0 \quad \mbox{and}\quad  \sum_{n=1}^{\infty}(X_n+Y_n)=1.
\end{equation}
In particular, the extreme points of $\mathcal{TH}(\Phi_{i},\Psi_{j};\alpha)$ are $\{h_n\}$ and $\{g_n\}$.
\end{theorem}

\begin{proof}
Let $f=h+\bar{g}\in \mathcal{TH}(\Phi_{i},\Psi_{j};\alpha)$ where $h$ and $g$ are given by \eqref{eq1.2}. Setting
\[X_n=\frac{p_{n}-\alpha u_{n}}{1-\alpha}A_n\, (n=2,3,\ldots)\quad \mbox{and}\quad Y_n=\frac{q_{n}-(-1)^{j-i}\alpha v_{n}}{1-\alpha}B_n\, (n=1,2,\ldots),\]
we note that $0\leq X_n\leq1$ ($n=2,3,\ldots$) and $0\leq Y_n\leq 1$ ($n=1,2,\ldots$) by Corollary \ref{cor2.4}. We define
\[X_1=1-\sum_{n=2}^{\infty}X_n-\sum_{n=1}^{\infty}Y_n.\]
By Theorem \ref{th2.3}, $X_1\geq0$ and $f$ can be expressed in the form \eqref{eq2.4}.

Conversely, for functions of the form \eqref{eq2.4} we obtain
\[f(z)=z-\sum_{n=2}^{\infty}\mu_n z^n+\sum_{n=1}^{\infty}\nu_n\bar{z}^{n},\]
where $\mu_n=(1-\alpha)X_n/(p_n-\alpha u_n)$ ($n=2,3,\ldots$) and $\nu_n=(1-\alpha)Y_n/(q_{n}-(-1)^{j-i}\alpha v_{n})$ ($n=1,2,\ldots$). Since
\[\sum_{n=2}^{\infty}\frac{p_{n}-\alpha u_{n}}{1-\alpha}\mu_n+\sum_{n=1}^{\infty}\frac{q_{n}-(-1)^{j-i}\alpha v_{n}}{1-\alpha}\nu_n=\sum_{n=2}^{\infty}X_n+\sum_{n=1}^{\infty}Y_n=1-X_1\leq1,\]
it follows that $f \in \mathcal{TH}(\Phi_{i},\Psi_{j};\alpha)$ by Theorem \ref{th2.3}.
\end{proof}

For harmonic functions of the form
\begin{equation}\label{eq2.5}
f(z)=z-\sum_{n=2}^{\infty}A_nz^{n}+\sum_{n=1}^{\infty}B_n\bar{z}^n \quad \mbox{and}\quad F(z)=z-\sum_{n=2}^{\infty}A'_nz^{n}+\sum_{n=1}^{\infty}B'_n\bar{z}^n,
\end{equation}
where $A_n, B_n, A'_n, B'_n\geq 0$, we define the product $\hat{*}$ of $f$ and $F$ as
\[(f\hat{*}F)(z)=z-\sum_{n=2}^{\infty}A_n A'_n z^n+\sum_{n=1}^{\infty}B_n B'_n \bar{z}^{n}=(F\hat{*}f)(z), \quad z \in \mathbb{D}.\]

Suppose that $\mathcal{I}$ and $\mathcal{J}$ are subclasses of $\mathcal{TH}$. We say that a class $\mathcal{I}$ is closed under $\hat{*}$ if $f\hat{*}F \in \mathcal{I}$ for all $f$, $F \in \mathcal{I}$. Similarly, the class $\mathcal{I}$ is closed under $\hat{*}$ with members of $\mathcal{J}$ if $f\hat{*}F \in \mathcal{I}$ for all $f \in \mathcal{I}$ and $F \in \mathcal{J}$. In general, the class $\mathcal{TH}(\Phi_{i},\Psi_{j};\alpha)$ is not closed under the product $\hat{*}$. The analytic function $f(z)=z-2z^2$ $(z\in \mathbb{D})$ belongs to $\mathcal{TH}(z+z^2/2,z;0)$, but $(f\hat{*}f)(z)=z-4z^2 \not\in \mathcal{TH}(z+z^2/2,z;0)$. However, we shall show that the class $\mathcal{TH}(\Phi_{i},\Psi_{j};\alpha)$ is closed under $\hat{*}$ with certain members of $\mathcal{TH}$.

\begin{theorem}\label{th2.9}
Suppose that $f,F \in \mathcal{TH}$ are given by \eqref{eq2.5} with $A'_n\leq 1$ and $B'_n\leq 1$. If $f \in \mathcal{TH}(\Phi_{i},\Psi_{j};\alpha)$ then $f\hat{*} F \in \mathcal{TH}(\Phi_{i},\Psi_{j};\alpha)$.
\end{theorem}

\begin{proof}
In view of Theorem \ref{th2.3}, it suffices to show that the coefficients of $f\hat{*}F$ satisfy condition \eqref{eq2.3}. Since
\begin{align*}
\sum_{n=2}^{\infty}\frac{p_{n}-\alpha u_{n}}{1-\alpha}A_n A'_n&+\sum_{n=1}^{\infty}\frac{q_{n}-(-1)^{j-i}\alpha v_{n}}{1-\alpha}B_n B'_n\\
&\leq \sum_{n=2}^{\infty}\frac{p_{n}-\alpha u_{n}}{1-\alpha}A_n+\sum_{n=1}^{\infty}\frac{q_{n}-(-1)^{j-i}\alpha v_{n}}{1-\alpha}B_n\leq 1,
\end{align*}
the result follows immediately.
\end{proof}

By imposing some restrictions on the coefficients of $\Phi_i$, it is possible for the class $\mathcal{TH}(\Phi_{i},\Psi_{j};\alpha)$ to be closed under the product $\hat{*}$, as seen by the following corollary.

\begin{corollary}\label{cor2.10}
If $f,F \in \mathcal{TH}(\Phi_{i},\Psi_{j};\alpha)$ with $p_n\geq 1$ $(n=2,3,\ldots)$ and $q_n \geq 1$ $(n=1,2,\ldots)$ then $f\hat{*}F \in \mathcal{TH}(\Phi_{i},\Psi_{j};\alpha)$.
\end{corollary}

\begin{proof}
Suppose that $f$ and $F$ are given by \eqref{eq2.5}. Using the hypothesis and Corollary \ref{cor2.4} it is easy to see that $A'_n\leq 1$ $(n=2,3,\ldots)$ and $B'_n\leq 1$ $(n=1,2,\ldots)$. The result now follows by Theorem \ref{th2.9}.
\end{proof}

In view of Corollary \ref{cor2.10}, it follows that the classes $\mathcal{TS}_H^*(\alpha)$ and $\mathcal{TK}_H(\alpha)$ are closed under the product $\hat{*}$. The next corollary shows that the class $\mathcal{TH}(\Phi_{i},\Psi_{j};\alpha)$ is preserved under certain integral transforms.

\begin{corollary}\label{cor2.11}
If $f \in \mathcal{TH}(\Phi_{i},\Psi_{j};\alpha)$ then $L_{\gamma}[f]$ and $G_\delta[f]$ belong to $\mathcal{TH}(\Phi_{i},\Psi_{j};\alpha)$, where $L_\gamma$ and $G_\delta$ are integral transforms defined by \eqref{eq1.5} and \eqref{eq1.6} respectively.
\end{corollary}

\begin{proof}
From the representations of $L_\gamma[f]$ and $G_\delta[f]$, it is easy to deduce that
\[L_\gamma[f](z)=f(z)\hat{*}\left(z-\sum_{n=2}^{\infty}\frac{\gamma+1}{\gamma+n} z^n+\sum_{n=1}^{\infty}\frac{\gamma+1}{\gamma+n}\bar{z}^{n}\right),\]
and
\[G_\delta[f](z)=f(z)\hat{*}\left(z-\sum_{n=2}^{\infty}\frac{1-\delta^n}{1-\delta}\frac{ z^n}{n}+\sum_{n=1}^{\infty}\frac{1-\delta^n}{1-\delta}\frac{\overline{z}^{n}}{n}\right),\]
where $z \in \mathbb{D}$. The proof of the corollary now follows by invoking Theorem \ref{th2.9}.
\end{proof}

The next two theorems provide sufficient conditions for the product $\hat{*}$ of $f \in \mathcal{TH}(\Phi_{i},\Psi_{j};\alpha)$ with certain members of $\mathcal{TH}$ associated with hypergeometric functions to be in the class $\mathcal{TH}(\Phi_{i},\Psi_{j};\alpha)$.
\begin{theorem}\label{th-hyper1}
Let $f \in \mathcal{TH}(\Phi_{i},\Psi_{j};\alpha)$ and $\Phi(z)=2z-\phi_1(z)+\overline{\phi_2(z)}$; $\phi_1$ and $\phi_2$ being given by \eqref{hyper1}. If $a_k, b_k>0$, $c_k>a_k+b_k$ for $k=1,2$ and if
\[F(a_1,b_1,c_1;1)+F(a_2,b_2,c_2;1)\leq 3,\]
then $f \hat{*}\Phi \in \mathcal{TH}(\Phi_{i},\Psi_{j};\alpha)$.
\end{theorem}

\begin{proof}
Writing $f=h+\overline{g}$ where $h$ and $g$ are given by \eqref{eq1.2}, note that
\[(f \hat{*}\Phi)(z)=z-\sum_{n=2}^{\infty} \frac{(a_1)_{n-1}(b_1)_{n-1}}{(c_1)_{n-1}(1)_{n-1}} A_n z^n+\sum_{n=1}^{\infty}\frac{(a_2)_{n}(b_2)_{n}}{(c_2)_{n}(1)_{n}} B_n \overline{z}^n,\quad z\in \mathbb{D}.\]
Applying Corollary \ref{cor2.4}, we deduce that
\begin{align*}
\sum_{n=2}^{\infty}\frac{p_{n}-\alpha u_{n}}{1-\alpha}\frac{(a_1)_{n-1}(b_1)_{n-1}}{(c_1)_{n-1}(1)_{n-1}} A_n &+\sum_{n=1}^{\infty}\frac{q_{n}-(-1)^{j-i}\alpha v_{n}}{1-\alpha}\frac{(a_2)_{n}(b_2)_{n}}{(c_2)_{n}(1)_{n}} B_n \\
&\leq \sum_{n=2}^{\infty}\frac{(a_1)_{n-1}(b_1)_{n-1}}{(c_1)_{n-1}(1)_{n-1}}+\sum_{n=1}^{\infty}\frac{(a_2)_{n}(b_2)_{n}}{(c_2)_{n}(1)_{n}}\\
&=F(a_1,b_1,c_1;1)+F(a_2,b_2,c_2;1)-2\leq 1.
\end{align*}
Theorem \ref{th2.3} now gives the desired result.
\end{proof}

\begin{theorem}\label{th-hyper2}
Let $f \in \mathcal{TH}(\Phi_{i},\Psi_{j};\alpha)$, $a_k, b_k>0$ and $c_k>a_k+b_k$ for $k=1,2$. Furthermore, if
\[F(a_1,b_1,c_1;1)+F(a_2,b_2,c_2;1)\leq 4,\]
then $f \hat{*}\Psi \in \mathcal{TH}(\Phi_{i},\Psi_{j};\alpha)$ where $\Psi(z)=2z-\psi_1(z)+\overline{\psi_2(z)}$; $\psi_1$ and $\psi_2$ being given by \eqref{hyper2}.
\end{theorem}

\begin{proof}
For $f=h+\overline{g}$ where $h$ and $g$ are given by \eqref{eq1.2}, we have
\[(f \hat{*}\Psi)(z)=z-\sum_{n=2}^{\infty} \frac{(a_1)_{n-1}(b_1)_{n-1}}{(c_1)_{n-1}(1)_{n}} A_n z^n+\sum_{n=2}^{\infty}\frac{(a_2)_{n-1}(b_2)_{n-1}}{(c_2)_{n-1}(1)_{n}} B_n \overline{z}^n,\quad z\in \mathbb{D}.\]
A simple calculation shows that
\begin{align*}
\sum_{n=2}^{\infty}\frac{p_{n}-\alpha u_{n}}{1-\alpha}\frac{(a_1)_{n-1}(b_1)_{n-1}}{(c_1)_{n-1}(1)_{n}} A_n  &+\sum_{n=2}^{\infty}\frac{q_{n}-(-1)^{j-i}\alpha v_{n}}{1-\alpha}\frac{(a_2)_{n-1}(b_2)_{n-1}}{(c_2)_{n-1}(1)_{n}} B_n \\
&\leq \sum_{n=2}^{\infty}\frac{(a_1)_{n-1}(b_1)_{n-1}}{(c_1)_{n-1}(1)_{n}} +\sum_{n=2}^{\infty}\frac{(a_2)_{n-1}(b_2)_{n-1}}{(c_2)_{n-1}(1)_{n}} \\
&\leq \frac{1}{2}\sum_{n=2}^{\infty}n\frac{(a_1)_{n-1}(b_1)_{n-1}}{(c_1)_{n-1}(1)_{n}} +\frac{1}{2}\sum_{n=2}^{\infty}n\frac{(a_2)_{n-1}(b_2)_{n-1}}{(c_2)_{n-1}(1)_{n}} \\
&=\frac{1}{2}(F(a_1,b_1,c_1;1)+F(a_2,b_2,c_2;1))-1\leq 1,
\end{align*}
using the hypothesis and Corollary \ref{cor2.4}. Hence $f \hat{*}\Psi \in \mathcal{TH}(\Phi_{i},\Psi_{j};\alpha)$ by Theorem \ref{th2.3}.
\end{proof}

We close this section by determining the convex combination properties of the members of the class $\mathcal{TH}(\Phi_{i},\Psi_{j};\alpha)$.

\begin{theorem}\label{th2.12}
The class $\mathcal{TH}(\Phi_{i},\Psi_{j};\alpha)$ is closed under convex combinations.
\end{theorem}

\begin{proof}
For $k=1,2,\ldots$ suppose that $f_{k}\in \mathcal{TH}(\Phi_{i},\Psi_{j};\alpha)$ where
\[f_k(z)=z-\sum_{n=2}^{\infty}a_{n,k}z^n+\sum_{n=1}^{\infty}b_{n,k}\bar{z}^n,\quad z\in \mathbb{D}.\]
For $\sum_{k=1}^{\infty}t_k=1$, $0\leq t_k\leq 1$, the convex combination of $f_k$'s may be written as
\[f(z)=\sum_{k=1}^{\infty}t_{k}f_{k}(z)=z-\sum_{n=2}^{\infty}\gamma_n z^n+\sum_{n=1}^{\infty}\delta_n\bar{z}^{n},\quad z\in \mathbb{D}\]
where $\gamma_n=\sum_{k=1}^{\infty}t_k a_{n,k}$ $(n=2,3,\ldots)$ and $\delta_n=\sum_{k=1}^{\infty}t_k b_{n,k}$ $(n=1,2,\ldots)$. Since
\begin{align*}
\sum_{n=2}^{\infty}\frac{p_{n}-\alpha u_{n}}{1-\alpha}\gamma_n&+\sum_{n=1}^{\infty}\frac{q_{n}-(-1)^{j-i}\alpha v_{n}}{1-\alpha}\delta_n\\
&=\sum_{k=1}^{\infty} t_k\left(\sum_{n=2}^{\infty}\frac{p_{n}-\alpha u_{n}}{1-\alpha}a_{n,k}+\sum_{n=1}^{\infty}\frac{q_{n}-(-1)^{j-i}\alpha v_{n}}{1-\alpha}b_{n,k}\right)\\
&\leq \sum_{k=1}^{\infty} t_k=1,
\end{align*}
it follows that $f \in \mathcal{TH}(\Phi_{i},\Psi_{j};\alpha)$ by Theorem \ref{th2.3}.
\end{proof}

\section{A particular case}\label{sec3}
For $0\leq \alpha<1$, set $\mathcal{U}_H(\alpha):=\mathcal{H}(z/(1-z)+\bar{z}/(1-\bar{z}),z;\alpha)$. Then $\mathcal{U}_{H}(\alpha)$ denote the set of all harmonic functions $f \in \mathcal{H}$ that satisfy $\RE f(z)/z>\alpha$ for $z \in \mathbb{D}$ and let $\mathcal{TU}_{H}(\alpha):=\mathcal{U}_{H}(\alpha)\cap \mathcal{TH}$. Applying Theorem \ref{th2.3} with $p_n-1=0=u_n$ $(n=2,3,\ldots)$, $q_n-1=0=v_n$ $(n=1,2,\ldots)$ and using the results of Section \ref{sec2}, we obtain

\begin{theorem}\label{th3.1}
Let the function $f=h+\bar{g}$ be such that $h$ and $g$ are given by \eqref{eq1.2} and $0\leq \alpha<1$. Then $f \in \mathcal{TU}_{H}(\alpha)$ if and only if
\[\sum_{n=2}^{\infty}\frac{A_n}{1-\alpha}+\sum_{n=1}^{\infty}\frac{B_n}{1-\alpha}\leq 1.\]
Furthermore, if $f \in \mathcal{TU}_{H}(\alpha)$ then $A_n\leq 1-\alpha$ $(n=2,3,\ldots)$, $B_n\leq 1-\alpha$ $(n=1,2,\ldots)$ and
\begin{equation}\label{eq3.1}
(1-B_1)|z|-(1-\alpha-B_1)|z|^2\leq|f(z)|\leq(1+B_1)|z|+(1-\alpha-B_1)|z|^2
\end{equation}
for all $z \in \mathbb{D}$. In particular, the range $f(\mathbb{D})$ contains the disk $|w|<\alpha$.

Moreover, the extreme points of the class $\mathcal{TU}_{H}(\alpha)$ are $\{h_n\}$ and $\{g_n\}$ where $h_1(z)=z$, $h_n(z)=z-(1-\alpha)z^n$ $(n=2,3,\ldots)$ and $g_n(z)=z+(1-\alpha)\bar{z}^{n}$ $(n=1,2,\ldots)$.
\end{theorem}

\begin{theorem}
Let $a_k, b_k>0$, $c_k>a_k+b_k$ for $k=1,2$. Then a necessary and sufficient condition for the harmonic function $\Phi(z)=2z-\phi_1(z)+\overline{\phi_2(z)}$ to be in the class $\mathcal{TU}_H(\alpha)$ is that
\[F(a_1,b_1,c_1;1)+F(a_2,b_2,c_2;1)\leq 3-\alpha,\]
where $\phi_1$ and $\phi_2$ are given by \eqref{hyper1}.
\end{theorem}

The upper bound given in \eqref{eq3.1} for $f \in \mathcal{TU}_{H}(\alpha)$ is sharp and equality occurs for the function $f(z)=z+B_1\bar{z}+(1-\alpha-B_1)\bar{z}^{2}$ for $B_1\leq 1-\alpha$.  In a similar fashion, comparable results to Corollary \ref{cor-hyper4} and Theorems \ref{th-hyper1}, \ref{th-hyper2} for the class $\mathcal{TU}_H(\alpha)$ may also be obtained. For further investigation of results regarding $\mathcal{TU}_{H}(\alpha)$, we need to prove the following simple lemma.

\begin{lemma}\label{lem3.2}
Let $f=h+\bar{g}\in \mathcal{H}$ where $h$ and $g$ are given by \eqref{eq1.1} with $B_1=g'(0)=0$. Suppose that $\lambda \in (0,1]$.
\begin{itemize}
  \item [(i)] If $\sum_{n=2}^{\infty}(|A_n|+|B_n|)\leq \lambda$ then $f \in \mathcal{U}_{H}(1-\lambda)$;
  \item [(ii)] If $\sum_{n=2}^{\infty}n(|A_n|+|B_n|)\leq \lambda$ then $f \in \mathcal{U}_{H}(1-\lambda/2)$ and is starlike of order $2(1-\lambda)/(2+\lambda)$.
\end{itemize}
The results are sharp.
\end{lemma}

\begin{proof}
Part (i) follows by Theorem \ref{th3.1}. For the proof of (ii), note that
\[\sum_{n=2}^{\infty}(|A_n|+|B_n|)\leq \frac{1}{2}\sum_{n=2}^{\infty} n(|A_n|+B_n|)\leq \frac{\lambda}{2}.\]
By part (i) of the lemma, $f \in \mathcal{U}_{H}(1-\lambda/2)$. The order of starlikeness of $f$ follows by \cite[Theorem 3.6]{sumit1}. The functions $z+\lambda \bar{z}^2$ and $z+\lambda \bar{z}^2/2$ show that the results in (i) and (ii) respectively are best possible.
\end{proof}

Using Corollary \ref{cor2.10}, Theorem \ref{th3.1} and Lemma \ref{lem3.2}(i), we obtain the following corollary.

\begin{corollary}
The class $\mathcal{TU}_{H}(\alpha)$ is closed under the product $\hat{*}$. In fact
\[\mathcal{TU}_{H}(\alpha)\hat{*}\mathcal{TU}_{H}(\beta)\subset \mathcal{TU}_{H}(1-(1-\alpha)(1-\beta))\]
for $\alpha, \beta \in [0,1)$.
\end{corollary}

A well-known classical result involving differential inequalities in univalent function theory is Marx Strohh\"{a}cker theorem \cite[Theorem 2.6(a), p. 57]{monograph} which states that if $f$ is an analytic function in $\mathbb{D}$ with $f(0)=0=f'(0)-1$ then
\[\RE\left(1+\frac{zf''(z)}{f'(z)}\right)>0\quad \Rightarrow\quad \RE \frac{zf'(z)}{f(z)}>\frac{1}{2}\quad \Rightarrow\quad \RE \frac{f(z)}{z}>\frac{1}{2}\quad (z\in \mathbb{D}).\]
The function $f(z)=z/(1-z)$ shows that all these implications are sharp. However, this theorem does not extend to univalent harmonic mappings, that is, if $f \in \mathcal{S}_{H}$ maps $\mathbb{D}$ onto a convex domain then it is not true in general that $\RE f(z)/z>0$ for all $z \in \mathbb{D}$.. To see this, consider the harmonic half-plane mapping
\[L(z)=\frac{z-\frac{1}{2}z^2}{(1-z)^2}+\overline{\frac{-\frac{1}{2}z^2}{(1-z)^2}},\quad z \in \mathbb{D}\]
which maps the unit disk $\mathbb{D}$ univalently onto the half-plane $\RE w>-1/2$. Figure \ref{fig1} shows that the function $L(z)/z$ does not have a positive real part in $\mathbb{D}$.

\begin{figure}[hb]
\centering
\includegraphics[width=3in]{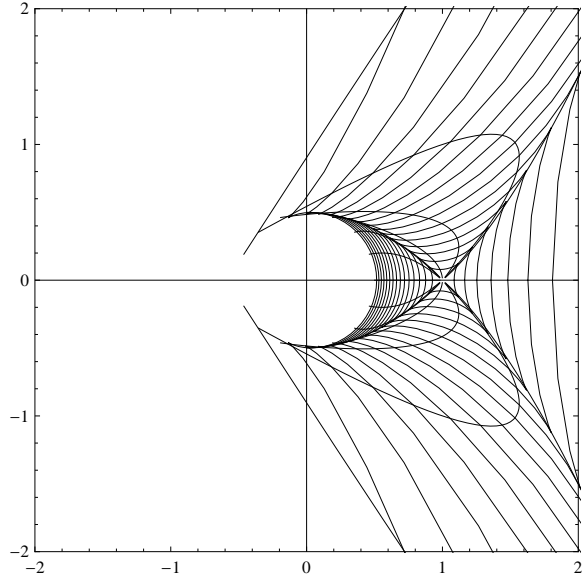}
\caption{Graph of the function $L(z)/z$.}\label{fig1}
\end{figure}

Denote by $\mathcal{TS}^{*0}_{H}(\alpha)$, $\mathcal{TK}^{0}_{H}(\alpha)$ and $\mathcal{TU}^{0}_{H}(\alpha)$, the classes consisting of functions $f$ in $\mathcal{TS}^{*}_{H}(\alpha)$, $\mathcal{TK}_{H}(\alpha)$ and $\mathcal{TU}_{H}(\alpha)$ respectively, for which $f_{\bar{z}}(0)=0$. The next theorem connects the relation between these three classes.

\begin{theorem}\label{th3.3}
For $0\leq \alpha <1$, the following sharp inclusions hold:
\begin{equation}\label{eq3.2}
\mathcal{TK}_{H}^{0}(\alpha)\subset \mathcal{TU}_{H}^{0}\left(\frac{3-\alpha}{2(2-\alpha)}\right);
\end{equation}
and
\begin{equation}\label{eq3.3}
\mathcal{TS}_{H}^{*0}(\alpha)\subset \mathcal{TU}_{H}^{0}\left(\frac{1}{2-\alpha}\right).
\end{equation}
\end{theorem}

\begin{proof}
Let $f=h+\bar{g} \in \mathcal{TH}$ where $h$ and $g$ are given by \eqref{eq1.2}. If $f \in \mathcal{TK}_{H}^{0}(\alpha)$ then
\[\sum_{n=2}^{\infty}n(A_n+B_n)\leq \frac{1}{2-\alpha}\sum_{n=2}^{\infty}n(n-\alpha)(A_n+B_n)\leq \frac{1-\alpha}{2-\alpha}\]
using \cite{jahangiriconvex}. By Lemma \ref{lem3.2}(ii) $f \in \mathcal{TU}_{H}^{0}((3-\alpha)/(2(2-\alpha)))$. Regarding the other inclusion, note that if $f \in \mathcal{TS}_{H}^{*0}(\alpha)$ then
\[\sum_{n=2}^{\infty}(A_n+B_n)\leq \frac{1}{2-\alpha}\sum_{n=2}^{\infty}(n-\alpha)(A_n+B_n)\leq \frac{1-\alpha}{2-\alpha}\]
by \cite[Theorem 2, p.\ 474]{jahangiristarlike}. This shows that $f \in \mathcal{TU}_{H}^{0}(1/(2-\alpha))$ by Lemma \ref{lem3.2}(i) as desired. The analytic functions $z-(1-\alpha)z^2/(2(2-\alpha))$ and $z-(1-\alpha)z^2/(2-\alpha)$ show that inclusions in \eqref{eq3.2} and \eqref{eq3.3} respectively are sharp.
\end{proof}

\begin{remark}
The proof of Theorem \ref{th3.3} shows that if $f \in \mathcal{TK}_{H}^{0}(\alpha)$ then $f$ is starlike of order $2/(5-3\alpha)$ by applying Lemma \ref{lem3.2}(ii). This gives the inclusion
\[\mathcal{TK}_{H}^{0}(\alpha)\subset \mathcal{TS}_{H}^{*0}\left(\frac{2}{5-3\alpha}\right).\]
It is not known whether this inclusion is sharp for $\alpha \in (0,1)$. However, if $\alpha=0$ then the inclusion $\mathcal{TK}_{H}^{0}(0)\subset \mathcal{TS}_{H}^{*0}(2/5)$ is sharp with the extremal function as $f(z)=z+\bar{z}^2/4$.
\end{remark}

The functions in the class $\mathcal{TU}_{H}(\alpha)$ need not be univalent in $\mathbb{D}$. The last theorem of this section determines the radius of univalence, starlikeness and convexity of the class $\mathcal{TU}^{0}_{H}(\alpha)$.

\begin{theorem}
The radius of univalence of the class $\mathcal{TU}_{H}^{0}(\alpha)$ is $1/(2(1-\alpha))$. This bound is also the radius of starlikeness of $\mathcal{TU}_{H}^{0}(\alpha)$. The radius of convexity of the class $\mathcal{TU}_{H}^{0}(\alpha)$ is $1/(4(1-\alpha))$.
\end{theorem}

\begin{proof}
Let $f=h+\bar{g}\in \mathcal{TU}_{H}^{0}(\alpha)$ where $h$ and $g$ are given by \eqref{eq1.2} and let $r \in (0,1)$ be fixed. Then $r^{-1}f(rz) \in \mathcal{TU}_{H}^{0}(\alpha)$ and we have
\[\sum_{n=2}^{\infty}n(A_n+B_n)r^{n-1}\leq \sum_{n=2}^{\infty}\left(\frac{A_n}{1-\alpha}+\frac{B_n}{1-\alpha}\right)\leq 1\]
provided $nr^{n-1}\leq1/(1-\alpha)$ which is true if $r \leq 1/(2(1-\alpha))$. In view of \cite[Theorem 1, p. 284]{silver}, $f$ is univalent and starlike in $|z|<1/(2(1-\alpha))$. Regarding the radius of convexity, note that
\[\sum_{n=2}^{\infty}n^2(A_n+B_n)r^{n-1}\leq \sum_{n=2}^{\infty}\left(\frac{A_n}{1-\alpha}+\frac{B_n}{1-\alpha}\right)\leq 1\]
provided $n^{2}r^{n-1}\leq1/(1-\alpha)$ which is true if $r \leq 1/(4(1-\alpha))$. The function $f(z)=z+(1-\alpha)\bar{z}^{2}$ shows that these bounds are sharp. This completes the proof of the theorem.
\end{proof}

\section*{Acknowledgements}
The research work of the first author is supported by research fellowship from Council of Scientific and Industrial Research (CSIR), New Delhi. The authors are thankful to the referee for her useful comments.

\end{document}